\documentclass[reqno]{amsart}
\usepackage{amsmath, amssymb}

\numberwithin{equation}{section}

\newtheorem{theorem}{Theorem}[section]
\newtheorem{lemma}{Lemma}[section]

\title[]{\bf   An abstract nonlinear evolution problem: Energy decay}

\author{Paul  A. Ogbiyele }
\address{Department of Mathematics, University of Ibadan, Nigeria.}
\email{paulogbiyele@yahoo.com, pa.ogbiyele@ui.edu.ng} 

 \keywords{Energy decay, Damping potential.}
 \subjclass[2010]{35L05, 35B40, 93D15}
\begin{document}

\begin{abstract}
In this paper, we consider energy decay estimates for  the following nonlinear evolution problem
\begin{equation}\small\nonumber
\begin{split}
[P(u_t(t))]_t + A u(t) + B(t , x , u_t(t)) =0,\quad t\in J=(0,\infty), 
\end{split}
\end{equation}
under suitable assumptions on the  self-adjoint operators $P$ and $B$  and the  divergence operator $A$. The work extends the earlier work of Marcati, Nakao, Levin and Pucci on abstract evolution equations to the case with possibly nonlinear damping. The technique also differs from earlier techniques. Here, we make use of a modified perturbed energy technique.
\end{abstract}

\maketitle

\section{Introduction}
Nakao\cite{Nak}, investigated the decay behavior of solutions to some nonlinear evolution equations of the form
\begin{equation}\nonumber
u^{\prime\prime}+Bu^\prime(t)+A(u)=f(t)
\end{equation}
and
\[Bu^\prime(t)+A(u)=f(t),\]
under the assumption that as $t\rightarrow \infty$ the function $f(t) \rightarrow 0$.  In \cite{Nak2}, using similar technique and under suitable assumptions on the operators $B$ and $A$, He extended the study to the case where $B$ is dependent of $t$.  More precisely, He considered the nonlinear evolution equation
\[u^{\prime\prime}+B(t)u^\prime(t)+A(u)=0.\]
The question of asymptotic stability was also studied by Pucci and Serrin in \cite{PusSer}, where they considered stability results for dissipative wave systems of the non-autonomous type. They extended the results in \cite{PusSer2} by considering  an abstract evolution equation of the form
\begin{equation}\nonumber
[P(u^\prime(t))]^\prime + A(u(t)) + Q(t , u^\prime(t)) + F(u(t)) =0,
\end{equation}
where $A$,  $F$, $P$ and $Q$ are nonlinear operators defined on appropriate Banach spaces. Using the technique introduced in \cite{PusSer}, they obtained asymptotic stability results that extend the  work of \cite{Mar, Nak}. The reader is referred to \cite{CheBou, LuoXia, MakBah, NikPig, Y} for other results on stability/energy decay of solutions to abstract evolution equations.
 In this paper, we consider abstract nonlinear evolution problem of the form
\begin{equation}\label{ev1}
\begin{split}
&[P(u_t(t))]_t + A u(t) + B(t , x , u_t(t)) =0,\quad t>0, \;\; x\in \mathbb{R}^n,
\end{split}
\end{equation}
where the operators $P$, $A$ and $B$ are nonlinear and defined on appropriate Banach spaces. The technique follows that of \cite{me, new} and some notations and assumptions were drawn from \cite{PusSer2}.

\section{Preliminaries}
Let $X$ and $V$ be real Banach spaces with  $\langle \cdot , \cdot \rangle_X$ and $\langle \cdot , \cdot \rangle_V$ the natural dual mapping pairings, and $X^\prime$, $V^\prime$ their respective duals. In addition, we assume that there exists a Banach space $W$ satisfying $X\hookrightarrow W \hookrightarrow V$ continuously such that
\[ \Vert u \Vert_W \leq k\Vert u\Vert_X, \]
for a positive constant $k$. 

We assume that $P:V\rightarrow V^\prime$ is non-negative definite and symmetric.  In addition, $P$ and $B$ satisfy the following conditions:
\begin{itemize}
\item[($A_1$)] $ \mathcal{P}^* (v) = \langle P(v) , v\rangle_V -\mathcal{P}(v)\geq 0 \quad \text{in}\quad V$, where the function $\mathcal{P^*}$ represents the Hamiltonian of $\mathcal{P}$ in classical mechanics, and there exist constants $p>0$, $k_0>0$ and $\ell>1$ such that
\[ (p+1)\langle P(v) , v\rangle - p\mathcal{P}(v)\leq k_0\Vert P(v)\Vert^{\ell^{\prime}}_{V^{\prime}}  \quad \text{for all}\quad v\in V.\]
\item[($A_2$)] $ |\langle P(u_t) ,\phi \rangle| \leq \Vert P(u_t)\Vert_{W^\prime} \Vert \phi\Vert_W$
and in addition there exist  nonnegative locally bounded functions $\delta(t)$ and $j(t)$ such that for positive constants $\ell, m$ satisfying $\ell\leq m$,  we have
\[ \Vert P(u_t)\Vert_{W^\prime} \leq (\delta(t))^{\frac{1}{\ell}} \Vert P(u_t)\Vert_{V^\prime}\]
and
\[\Vert P(v)\Vert^{\ell^\prime}_{V^\prime} \leq j(t)\langle B(t , x , u_t), u_t\rangle^{\ell/m}.\]
\item[($A_3$)] $|\langle B(t , x, u_t) ,\phi \rangle| \leq \Vert B(t , x , u_t )\Vert_{W^{\prime}} \Vert \phi\Vert_{W}$
and $B\in C(S , W^\prime)$ where $S$ is a given subset of $J\times\Omega \times V$, and there exists a non-negative locally bounded function $\eta(t)$ such that 
\[ \Vert B(t , x, u_t)\Vert_{W^\prime} \leq (\eta(t))^\frac{1}{m} \langle B(t , x , u_t), u_t\rangle^{\frac{1}{m^\prime}}.\]
\end{itemize}
To define the energy function $\mathcal{E}(u) $ associated to (\ref{ev1}), we introduce the real-valued $C^1$ potential functions $\mathcal{A} : X\rightarrow \mathbb{R}$, $\mathcal{P} : V \rightarrow \mathbb{R}$,  where the operator $A:X\rightarrow W$ is the Frech\'{e}t derivative of  $\mathcal{A}$  with respect to $u$ and $\mathcal{P}$ is the Frech\'{e}t derivative of  $P$. In addition $\mathcal{A}(0) = \mathcal{P}(0)=0$.

For the operator $A$, we have the following assumption:
\begin{itemize}
\item[($A_4$)] $\mathcal{A}(u)\geq 0$  and there exist positive  constants   $c_0$ and $q\leq \ell$ such that
\[ q\mathcal{A}(u) \leq \langle A(u) , u\rangle\quad \text{and}\quad \Vert u\Vert^q_X \leq c_0 \mathcal{A}(u)   \quad \forall u \in X.\]
\end{itemize} 
Define $K=C(J: X)\cap C^1(J: V)$, then by a strong solution $u\in K$ of (\ref{ev1}), we mean a function $u\in K$ with  $u_0\in X$, $v_0\in V$ that satisfies
\begin{itemize}
\item[(a)] $u\in K$ and $(t , x, u_t(t))\in S$ for a.a $\; t\in J$,
 \item[(b)] the following distribution relation
\begin{equation}\label{dist} \langle P(u_t(s)), \phi(s)\rangle_V \Bigr|^t_0 = \int^t_0 \Bigl[ \langle P(u_t(s)) , \phi_t(s) \rangle_V -\langle A(u(s)) , \phi(s) \rangle_X - \langle B(s , x, u_t(s)) , \phi(s) \rangle_W \Bigr] ds
 \end{equation} 
holds for all $\phi\in K$, $t\in J$ with the associated total energy function
\begin{equation}\label{energy}
 \mathcal{E}(t) =\mathcal{P}^*(u_t(t)) + \mathcal{A}(u(t)),
\end{equation}
satisfying the classical conservation balance
\begin{equation}\nonumber
 \mathcal{E}(t) = \mathcal{E}(0) - \int^t_0 \langle B(s , x, u_t(s)) , u_t(s)\rangle ds \quad \text{for}\quad t\in J.
 \end{equation}
\end{itemize}
Note: When $S=J\times Y$ and $B:S\rightarrow Y^\prime$ where $Y$ is a Banach space with natural pairing $\langle\cdot , \cdot\rangle_Y$ and satisfying the continuous inclusion $X\hookrightarrow W \hookrightarrow Y\hookrightarrow V$, the distribution relation (\ref{dist}) is clarified with 
$\langle\cdot , \cdot\rangle_W$ replaced by $\langle\cdot , \cdot\rangle_Y$. In addition, the embedding $W \hookrightarrow Y\hookrightarrow V$ is not   compact in the case of unbounded domains.

Furthermore, we have the following monotonicity and regularity condition: 
\begin{itemize}
\item[($A_5$)]The function $\mathcal{E}(t)$ is nonincreasing and absolutely continuous on $J$ for strong solution $u$ of (\ref{ev1}).

\end{itemize}
Consequently, from ($A_2$), we have 
$$\Vert P(v)\Vert^{\ell^\prime}_{V^\prime} \leq j(t) \bigl[-\mathcal{E}^{\prime}(t)\bigr]^{\ell/m}$$
and ($A_3$) gives 
$$ \Vert B(t , x, u_t)\Vert_{W^\prime} \leq (\eta(t))^\frac{1}{m}  \bigl[-\mathcal{E}^{\prime}(t)\bigr]^{\frac{1}{m^\prime}}.$$
\begin{lemma}
Assume that the conditions ($A_1)$-($A_5$) hold. Let $u$ be a strong solution of (\ref{ev1}), then, the following is satisfied:
\begin{equation}\label{e0}
\mathcal{E}^{\prime}(t) \leq 0 \quad  \text{a.e. in J} \quad \text{and}\quad 0\leq \mathcal{E}(t)\leq \mathcal{E}(0) \quad  \text{in J}.
\end{equation}
Moreover,
\begin{equation}\nonumber
0\leq \mathcal{A}(u(t)) + \mathcal{P}(u_t(t)) \leq \mathcal{E}(0) \quad \text{in J}\quad \text{and} \quad -\mathcal{E}^\prime(t)\in L^1(J).
\end{equation}
\end{lemma}
The Lemma above is a partial modification of that of  Pucci and Serrin \cite{PusSer} and the proof  follows directly from ($A_1$), ($A_4$), (\ref{energy}) and (\ref{e0}). 
We state without proof the existence result of (\ref{ev1}).  
\begin{theorem}  For initial data $u_0\in X$ and $u_1\in V$, there exists  a unique solution $u$ of (\ref{ev1}) such that
\begin{equation}\nonumber
u\in C(\mathbb{R}_+; X)\cap C^1(\mathbb{R}_+; V).
\end{equation}
\end{theorem}
The proof follows that of  \cite{Nak3, OgbAra, Tsu}.
\section{Energy Decay}
For the energy decay, we define the following functions
\begin{equation}\label{evolH} \mathcal{H}(t) =\lambda(t) \mathcal{E}(t) + \mu \alpha(t)\mathcal{E}^r(t)\langle P(u_t), u\rangle
\end{equation}
and
\begin{equation}\label{NewG}
\mathcal{G}(t)=\mathcal{H}(t) + \nu \alpha(t) \delta^{\frac{1}{\ell}}(t) \mathcal{E}^{r+\frac{1}{q} +\frac{1}{\ell^{\prime}}}(t),
\end{equation}
where  the positive functions $\lambda(t)$ and $\alpha(t)$  are used to compensate for the compactness in the case of unbounded domains ($\lambda(t)$, $\alpha(t)$ and $\delta(t)$ are constants in the case of bounded domains) and $r$ is a positive constant satisfying $r+\frac{1}{q} +\frac{1}{\ell^{\prime}}\geq 1$. In addition, we state the following assumptions on $\lambda(t)$, $\alpha(t)$, $\delta(t)$, $\eta(t)$ and $j(t)$:
\begin{itemize}
\item[($B_1$)] $\lambda(t)\geq \alpha(t) \max\{\eta^{\frac{m^\prime}{m}}(t) , \delta^{\frac{1}{\ell}}(t)\} $, 
\item[($B_2$)] $ \bigl( \frac{ \alpha(t)   \delta^{\frac{1}{\ell}(t)}}{\lambda(t)}\bigr)^{\prime} \leq 0$
\item[($B_3$)] $ \Bigl(\frac{\lambda(t)}{\alpha(t)} \Bigl\vert\Bigl( \frac{\alpha(t)}{\lambda(t)} \Bigr)^{\prime}\Bigr\vert\Bigr) \delta^{\frac{1}{\ell}}(t)\leq c_{_{\alpha\lambda}}$
\item[($B_4$)] $\frac{ j^{\frac{m}{(m-\ell)}}(t)}{\eta^{\frac{m^\prime}{m}}(t)} \leq c_{j\eta}$\\
 where $ c_{j\eta}\; and \; c_{\alpha\lambda}$ are positive constants.
\end{itemize}
\begin{theorem}\label{thm2}
Let $u$ be a solution of the problem (\ref{ev1}). Assume that the conditions ($A_1$)-($A_5$) and ($B_1$)-($B_4$)  hold and $u_0\in X$, $u_1\in V$. Then, for $t\geq 0$, there exist positive constants $C$, $C^*$ and $C^{**}$ such that the energy of solution to (\ref{ev1}), satisfies
\begin{align}
\mathcal{E}(t)\leq C \Bigl[1+  \int^t_0 \frac{\alpha(s)}{\lambda(s)} ds \Bigr]^{\frac{-\ell}{m-\ell}}\qquad \text{for $m>\ell$}
\end{align}
and 
\[\mathcal{E}(t) \leq C^* exp\Bigl[-C^{**}\int^t_0  \frac{\alpha(s)}{\lambda(s)} ds \Bigr] \qquad \text{for $m=\ell$}. \]
\end{theorem}

We state and prove the following lemmas, that will be needed in the proof of  Theorem \ref{thm2}.
\begin{lemma}\label{Abslem}
Let $u$ be a solution of (\ref{ev1}) and assume that the conditions ($A_2$), ($A_4$) and ($B_1$) hold.  Then, there exist positive constants $k_1$ and $k_2$ such that the energy function $\mathcal{G}(t)$ satisfies
\begin{equation}\nonumber
k_1\mathcal{E}(t) \leq \frac{1}{\lambda(t)}\mathcal{G}(t) \leq k_2 \mathcal{E}(t).
\end{equation}
\end{lemma}
\begin{proof}
From (\ref{NewG}), we have 
\begin{equation}\label{NewGHE}
\begin{split}
\mathcal{G}(t)=\lambda(t) \mathcal{E}(t) + \mu \alpha(t)\mathcal{E}^r(t)\langle P(u_t), u\rangle +\nu \alpha(t) \delta^{\frac{1}{\ell}}(t) \mathcal{E}^{r+\frac{1}{q} +\frac{1}{\ell^{\prime}}}(t).
\end{split}
\end{equation}
Application of H\"{o}lder, Sobolev inequalities and assumption ($A_2$), ($A_4$) together with (\ref{energy}) to the second term in (\ref{NewGHE}), gives
\begin{equation}\label{Abs35}
\begin{split}
 \mu \alpha(t)\mathcal{E}^r(t)\langle P(u_t), u\rangle\leq& \mu \alpha(t)\mathcal{E}^r(t) \Vert P(u_t)\Vert_{W^\prime} \Vert u\Vert_W\\
 \leq &  \mu k \alpha(t)\delta^{\frac{1}{\ell}}(t)\mathcal{E}^r(t)\Vert P(u_t)\Vert_{V^\prime}\Vert u \Vert_X\\
  \leq &  \mu c_1 \alpha(t)\delta^{\frac{1}{\ell}}(t)\mathcal{E}^{r +\frac{1}{\ell^{\prime}}+\frac{1}{q}} (t),\end{split}
 \end{equation}
 where $c_1=c_1(k , c_0 , q)$ and $r+\frac{1}{q} +\frac{1}{\ell^{\prime}}\geq 1$. Then, using (\ref{Abs35}), (\ref{e0}) and condition ($B_1$), we have
\begin{equation}\nonumber
\begin{split}
|\frac{1}{\lambda(t)}\mathcal{G}(t)- \mathcal{E}(t)| &\leq  (\mu c_1 + \nu)\frac{\alpha(t)\delta^{\frac{1}{\ell}}(t)}{\lambda(t)} \mathcal{E}(t) \mathcal{E}^{(r +\frac{1}{\ell^{\prime}} +\frac{1}{q} -1)}(t)\\&
\leq  ( \mu c_1 +\nu) c^*_{_E} \mathcal{E}(t),
\end{split}
\end{equation} 
where $c^*_{_E} =\mathcal{E}^{(r+\frac{1}{q} +\frac{1}{\ell^{\prime}} -1)}(0)$. For $\mu$, $\nu$ small enough, we have the desired estimate.
\end{proof}

\begin{lemma}
Assume that the conditions ($A_3$), ($A_4$) hold. Let $u$ be a solution of (\ref{ev1}), then the function $\mathcal{H}$ satisfies
\begin{align*}
 \frac{d}{d t}\langle P(u_t) , u \rangle \leq& \langle P(u_t) , u_t \rangle+p\mathcal{P}^*(u_t) -\bigl(\langle A(u) , u\rangle -p\mathcal{A}(u)\bigr)\\&  +  c_2 (\eta(t))^\frac{1}{m} \bigl[-\mathcal{E}^{\prime}(t)\bigr]^{\frac{1}{m^\prime}}  \mathcal{E}^{\frac{1}{q}}(t)-p\mathcal{E}(t).
\end{align*}
\end{lemma}

\begin{proof}
From the derivative of  (\ref{dist}) and assumption ($A_3$), ($A_4$) and the identity (\ref{energy}), we have
\begin{equation}\nonumber
\begin{split}
\frac{d}{d t}\langle P(u_t) , u \rangle =&\langle P(u_t) , u_t \rangle -\langle A(u) , u\rangle -\langle B(t , x , u_t), u \rangle  \\&+ p \bigl(\mathcal{P}^*(u_t) + \mathcal{A}(u) -\mathcal{E}(t)\bigr)
 \\
\leq&  \langle P(u_t) , u_t \rangle+p\mathcal{P}^*(u_t) -\bigl(\langle A(u) , u\rangle -p\mathcal{A}(u)\bigr)\\&  + \Vert B(t , x , u_t)\Vert_{W^\prime} \Vert u\Vert_{W}   -p\mathcal{E}(t)
\\  \leq & \langle P(u_t) , u_t \rangle+p\mathcal{P}^*(u_t) -\bigl(\langle A(u) , u\rangle -p\mathcal{A}(u)\bigr)\\&  +  c_2 (\eta(t))^\frac{1}{m} \bigl[-\mathcal{E}^{\prime}(t)\bigr]^{\frac{1}{m^\prime}}  \mathcal{E}^{\frac{1}{q}}(t)-p\mathcal{E}(t).
\end{split}
\end{equation}

\end{proof}

\begin{proof}[Proof of Theorem \ref{thm2}]
Differentiating the function $\mathcal{H}(t)$ in (\ref{evolH}), we have
\begin{equation}\label{Abs39}
\begin{split}
\mathcal{H}^{\prime}(t) =&\lambda(t)\mathcal{E}^{\prime}(t)  +\lambda^{\prime}(t) \mathcal{E}(t) + \mu \alpha^{\prime}(t) \mathcal{E}^r(t)\langle P(u_t), u\rangle + \mu\alpha(t) \mathcal{E}^r(t) \frac{d}{dt} \langle P(u_t), u\rangle  \\&+\mu r\alpha(t) \mathcal{E}^{r-1}(t)\mathcal{E}^{\prime}(t)\langle P(u_t), u\rangle.
\end{split}
\end{equation}
For ease of estimation, we re-expressed the second and third term of (\ref{Abs39}) in terms of $\mathcal{H}(t)$ and $\langle P(u_t), u\rangle$. Hence, we have
\begin{equation}\label{Abs310}
\lambda^{\prime}(t) \mathcal{E}(t) + \mu \alpha^{\prime}(t) \mathcal{E}^r(t)\langle P(u_t), u\rangle=\frac{\lambda^{\prime}(t)}{\lambda(t)} \mathcal{H}(t)  + \mu\bigl(\frac{\alpha(t)}{\lambda(t)} \bigr)^{\prime} \lambda(t) \mathcal{E}^r(t)\langle P(u_t), u\rangle.
\end{equation}
The second term on the right of (\ref{Abs310}) can be re-estimated using ($A_2$), ($A_4$),  ($B_3$),(\ref{energy}) and (\ref{e0}) as follows:
\begin{equation}\nonumber
\begin{split}
\mu&\bigl(\frac{\alpha(t)}{\lambda(t)} \bigr)^{\prime} \lambda(t) \mathcal{E}^r(t)\langle P(u_t), u\rangle\\
 \leq & \mu k\Bigl\vert \bigl(\frac{\alpha(t)}{\lambda(t)} \bigr)^{\prime} \Bigr\vert \lambda(t)\delta^{\frac{1}{\ell}}(t)\mathcal{E}^r(t)\Vert P(u_t)\Vert_{V^\prime}\Vert u\Vert_X\\ \leq & \mu c_1  \Bigl(\frac{\lambda(t)}{\alpha(t)} \Bigl\vert\Bigl( \frac{\alpha(t)}{\lambda(t)} \Bigr)^{\prime}\Bigr\vert\Bigr)\delta^{\frac{1}{\ell}}(t)\alpha(t)\mathcal{E}^r(t)\Vert P(u_t)\Vert_{V^\prime}  \mathcal{A}(u)^{\frac{1}{q}}\\
\leq &
 \mu \epsilon_1 c_{_{\alpha\lambda}}c_{_1} \alpha(t)\mathcal{E}^r(t)\Vert P(u_t)\Vert^{\ell^{\prime}}_{V^\prime} + \mu c(\epsilon_1) c_{_{\alpha\lambda}}c_{_1} \alpha(t)\mathcal{E}^r(t) \mathcal{A}(u)^{\frac{\ell}{q}}\\
 \leq &
 \mu \epsilon_1 c_{_{\alpha\lambda}}c_{_1} \alpha(t)\mathcal{E}^r(t)\Vert P(u_t)\Vert^{\ell^{\prime}}_{V^\prime} + \mu c(\epsilon_1) c_{_{\alpha\lambda}}c_{_3} \alpha(t)\mathcal{E}^{r}(t) \mathcal{A}(u),
\end{split}
\end{equation}
where  $c_3=c_3(c_1, \mathcal{E}(0))$.
For the fourth term, using ($A_2$), ($A_4$)  and (\ref{energy}), we obtain
\begin{equation}\nonumber
\begin{split}
\mu r \alpha(t) \mathcal{E}^{r-1} \mathcal{E}^{\prime}(t)\langle P(u_t) , u \rangle &\leq \mu r \alpha(t) \mathcal{E}^{r-1}(t) (-\mathcal{E}^{\prime}(t))|\langle P(u_t) , u \rangle| 
\\&\leq  \mu r k\delta^{\frac{1}{\ell}}(t) \alpha(t) \mathcal{E}^{r-1} (-\mathcal{E}^{\prime}(t))\Vert P(u_t)\Vert_{V^\prime} \Vert u\Vert_X\\&
\leq  \mu c_{_4} r\delta^{\frac{1}{\ell}}(t) \alpha(t) \mathcal{E}^{r-1+\frac{1}{q} +\frac{1}{\ell^{\prime}}}(t) (-\mathcal{E}^{\prime}(t))\\&
\leq \frac{- \mu c_{_4} r}{r+\frac{1}{q} +\frac{1}{\ell^{\prime}}} \delta^{\frac{1}{\ell}}(t) \alpha(t)\frac{d}{dt}\bigl(\mathcal{E}^{r+\frac{1}{q} +\frac{1}{\ell^{\prime}}}(t) \bigr).
\end{split}
\end{equation}
Therefore,
\begin{equation}\label{Abs313}
\begin{split}
\mathcal{H}^{\prime}(t) \leq &\frac{\lambda^{\prime}(t)}{\lambda(t)} \mathcal{H}(t)  + \mu\alpha(t)\mathcal{E}^r(t) \bigl( \epsilon_1 c_{_{\alpha\lambda}}c_{_1} \Vert P(u_t)\Vert^{\ell^{\prime}}_{V^\prime} + \langle P(u_t) , u_t \rangle+p\mathcal{P}^*(u_t) \bigr) \\&- \mu\alpha(t) \mathcal{E}^r(t) \bigl(\langle A(u) , u\rangle -(p+c(\epsilon_1) c_{_{\alpha\lambda}}c_{_3})\mathcal{A}(u)\bigr) -\mu p\alpha(t) \mathcal{E}^{r+1}(t) +\lambda(t)\mathcal{E}^{\prime}(t) \\&-\frac{ \mu c_{_4} r \delta^{\frac{1}{\ell}}(t) \alpha(t)}{r+\frac{1}{q} +\frac{1}{\ell^{\prime}}} \frac{d}{dt}\bigl(\mathcal{E}^{r+\frac{1}{q} +\frac{1}{\ell^{\prime}}}(t) \bigr) + \mu c_2 \alpha(t) (\eta(t))^\frac{1}{m} \bigl[-\mathcal{E}^{\prime}(t)\bigr]^{\frac{1}{m^\prime}}   \mathcal{E}^{r+\frac{1}{q}}(t).
\end{split}
\end{equation}
Applying Young's inequality  to the last term in (\ref{Abs313}), gives
\begin{equation}\label{Abs314}
\begin{split}
  \mu &c_2 \alpha(t) (\eta(t))^\frac{1}{m} \bigl[-\mathcal{E}^{\prime}(t)\bigr]^{\frac{1}{m^\prime}}   \mathcal{E}^{r+\frac{1}{q}}(t)
\\&\leq - \epsilon_2  c_2\mu\alpha(t) \eta^{\frac{m^\prime}{m}}(t) \mathcal{E}^{\prime}(t) +c(\epsilon_2 ) c_2\mu\alpha(t)  \mathcal{E}^{m(r+\frac{1}{q})}(t)\\
 &\leq  - \epsilon_2  c_2\mu\alpha(t) \eta^{\frac{m^\prime}{m}}(t) \mathcal{E}^{\prime}(t) + c(\epsilon_2 ) c_2c_{_E}\mu\alpha(t) \mathcal{E}^{r+1}(t),
\end{split}
\end{equation}
where $c_{_E}= \mathcal{E}^{r(m-1)+\frac{m-q}{q}}(0) $. 
Substituting (\ref{Abs314}) in (\ref{Abs313}) and applying condition ($B_1$) to the third to the last term on the right side of (\ref{Abs314}), we obtain
\begin{equation}\label{Abs315}
\begin{split}
\mathcal{H}^{\prime}(t) \leq &\frac{\lambda^{\prime}(t)}{\lambda(t)} \mathcal{H}(t)  + \mu\alpha(t)\mathcal{E}^r(t) \bigl( \epsilon_1 c_{_{\alpha\lambda}}c_{_1} \Vert P(u_t)\Vert^{\ell^{\prime}}_{V^\prime} +(p+1) \langle P(u_t) , u_t \rangle-p\mathcal{P}(u_t) \bigr) \\&
- \mu\alpha(t) \mathcal{E}^r(t) \bigl(\langle A(u) , u\rangle -(p+c(\epsilon_1) c_{_{\alpha\lambda}}c_{_3})\mathcal{A}(u)\bigr)+\alpha(t) \eta^{\frac{m^\prime}{m}}(t)\bigl( 1- \epsilon_2  c_2\mu\bigr)\mathcal{E}^{\prime}(t)
 \\&
-\frac{ \mu c_{_4} r \delta^{\frac{1}{\ell}}(t) \alpha(t)}{r+\frac{1}{q} +\frac{1}{\ell^{\prime}}} \frac{d}{dt}\bigl(\mathcal{E}^{r+\frac{1}{q} +\frac{1}{\ell^{\prime}}}(t) \bigr) -\mu \alpha(t) \bigl(p - c(\epsilon_2 ) c_2c_{_E}  \bigr) \mathcal{E}^{r+1}(t).
 \end{split}
\end{equation}
Using the integrating factor $\frac{1}{\lambda(t)}$ and  condition ($A_1$), ($A_4$) with $q\geq p+c(\epsilon_1) c_{_{\alpha\lambda}}c_{_3}$ in (\ref{Abs315}), we get the following estimate:
\begin{equation}\label{Abs316}
\begin{split}
\Bigl(\frac{1}{\lambda(t)} \mathcal{H}(t)\Bigr)^{\prime} \leq &   \mu ( \epsilon_1 c_{_{\alpha\lambda}}c_{_1} +k_0)\frac{\alpha(t)}{\lambda(t)}\mathcal{E}^r(t) \Vert P(u_t)\Vert^{\ell^{\prime}}_{V^\prime}  +\frac{\alpha(t)}{\lambda(t)}\eta^{\frac{m^\prime}{m}}(t)\bigl( 1- \epsilon_2  c_2\mu\bigr)\mathcal{E}^{\prime}(t)
 \\&
-\frac{ \mu c_{_3} r \delta^{\frac{1}{\ell}(t)}}{r+\frac{1}{q} +\frac{1}{\ell^{\prime}}} \frac{\alpha(t)}{\lambda(t)} \frac{d}{dt}\bigl(\mathcal{E}^{r+\frac{1}{q} +\frac{1}{\ell^{\prime}}}(t) \bigr) -\mu \frac{\alpha(t)}{\lambda(t)} \bigl(p - c(\epsilon_2 ) c_2c_{_E}  \bigr) \mathcal{E}^{r+1}(t).
 \end{split}
\end{equation}
The second to the last term in (\ref{Abs316}) can be re-expressed using condition ($B_2$) as
\begin{equation}\label{Abs317}
\begin{split}
-\frac{ \mu c_{_4} r}{r+\frac{1}{q} +\frac{1}{\ell^{\prime}}}  \frac{ \alpha(t)  \delta^{\frac{1}{\ell}(t)}}{\lambda(t)} \frac{d}{dt}\bigl(\mathcal{E}^{r+\frac{1}{q} +\frac{1}{\ell^{\prime}}}(t) \bigr) =& -\frac{ \mu c_{_4} r}{r+\frac{1}{q} +\frac{1}{\ell^{\prime}}} \frac{d}{dt} \bigl( \frac{ \alpha(t)   \delta^{\frac{1}{\ell}(t)}}{\lambda(t)} \mathcal{E}^{r+\frac{1}{q} +\frac{1}{\ell^{\prime}}}(t) \bigr) \\ &
+ \frac{ \mu c_{_4} r}{r+\frac{1}{q} +\frac{1}{\ell^{\prime}}} \bigl( \frac{ \alpha(t)   \delta^{\frac{1}{\ell}(t)}}{\lambda(t)}\bigr)^{\prime}  \mathcal{E}^{r+\frac{1}{q} +\frac{1}{\ell^{\prime}}}(t)\\\leq &
 -\frac{ \mu c_{_4} r}{r+\frac{1}{q} +\frac{1}{\ell^{\prime}}} \frac{d}{dt} \bigl( \frac{ \alpha(t)   \delta^{\frac{1}{\ell}(t)}}{\lambda(t)} \mathcal{E}^{r+\frac{1}{q} +\frac{1}{\ell^{\prime}}}(t) \bigr).
\end{split}
\end{equation}
Then,  using  assumption ($A_2$) and Young's inequality,  the first term on the right side of (\ref{Abs317}), gives the following:
\begin{equation}\label{Abs318}
\begin{split}
 \mu& ( \epsilon_1 c_{_{\alpha\lambda}}c_{_1} +k_0)\frac{\alpha(t)}{\lambda(t)} \mathcal{E}^r(t) \Vert P(u_t)\Vert^{\ell^{\prime}}_{V^\prime}\\& \leq 
  \mu ( \epsilon_1 c_{_{\alpha\lambda}}c_{_1} +k_0)\frac{\alpha(t)}{\lambda(t)}j(t) \mathcal{E}^r(t) \bigl[-\mathcal{E}^{\prime}(t)\bigr]^{\ell/m}   \\& \leq 
   -\mu\epsilon_3 ( \epsilon_1 c_{_{\alpha\lambda}}c_{_1} +k_0)\frac{\alpha(t) j^{\frac{m}{(m-\ell)}}(t)}{\lambda(t)} \mathcal{E}^{\prime}(t) + \mu c(\epsilon_3) ( \epsilon_1 c_{_{\alpha\lambda}}c_{_1} + k_0)\frac{\alpha(t)}{\lambda(t)} \mathcal{E}^{\frac{r m}{(m-\ell)}}(t).
 \end{split}
\end{equation}
Substituting (\ref{Abs317}) and (\ref{Abs318}) in (\ref{Abs316}), we obtain
\begin{equation}\nonumber
\begin{split}
\Bigl(&\frac{1}{\lambda(t)} \mathcal{H}(t) + \frac{ \mu c_{_4} r}{r+\frac{1}{q} +\frac{1}{\ell^{\prime}}}  \frac{ \alpha(t)  \delta^{\frac{1}{\ell}}(t)}{\lambda(t)} \mathcal{E}^{r+\frac{1}{q} +\frac{1}{\ell^{\prime}}}(t) \Bigr)^{\prime}\\ \leq &  \mu c(\epsilon_3) ( \epsilon_1 c_{_{\alpha\lambda}}c_{_1} +k_0)\frac{\alpha(t)}{\lambda(t)} \mathcal{E}^{\frac{r m}{(m-\ell)}}(t)  
 -\mu \frac{\alpha(t)}{\lambda(t)} \bigl(p - c(\epsilon_2 ) c_2c_{_E}  \bigr) \mathcal{E}^{r+1}(t) \\&+\frac{\alpha(t)}{\lambda(t)} \eta^{\frac{m^\prime}{m}}(t) \Bigl[ \bigl( 1- \epsilon_2  c_2\mu\bigr) - \mu\epsilon_3 ( \epsilon_1 c_{_{\alpha\lambda}}c_{_1} +k_0)\frac{ j^{\frac{m}{(m-\ell)}}(t)}{\eta^{\frac{m^\prime}{m}}(t)} \Bigr]\mathcal{E}^{\prime}(t).
 \end{split}
\end{equation}
Set $r=(m-\ell)/\ell$, and use condition ($B_4$) to get
\begin{equation}\label{cases}
\begin{split}
\Bigl(&\frac{1}{\lambda(t)} \mathcal{H}(t) + \frac{ \mu c_{_4} q(m -\ell)}{(q(m-1) + \ell)}  \frac{ \alpha(t)  \delta^{\frac{1}{\ell}}(t)}{\lambda(t)} \mathcal{E}^{\frac{q(m-1) +\ell}{q\ell}}(t) \Bigr)^{\prime}\\ \leq &    
 -\mu \frac{\alpha(t)}{\lambda(t)} \bigl[p - c(\epsilon_2 ) c_2c_{_E} - c(\epsilon_3)( \epsilon_1 c_{_{\alpha\lambda}}c_{_1} +k_0) \bigr] \mathcal{E}^{\frac{m}{\ell}}(t) \\&+\frac{\alpha(t)}{\lambda(t)} \eta^{\frac{m^\prime}{m}}(t) \Bigl[ \bigl( 1- \epsilon_2  c_2\mu\bigr) - \mu\epsilon_3 ( \epsilon_1 c_{_{\alpha\lambda}}c_{_1} +k_0) c_{_{j\eta}}\Bigr]\mathcal{E}^{\prime}(t).
 \end{split}
\end{equation}
Case 1:\\
From (\ref{NewG}), we have that 
\[\frac{1}{\lambda(t)}G(t)=\frac{1}{\lambda(t)} \mathcal{H}(t)+ \frac{ \mu c_{_3} q(m -\ell)}{(q(m-1) + \ell)}  \frac{ \alpha(t)  \delta^{\frac{1}{\ell}}(t)}{\lambda(t)} \mathcal{E}^{\frac{q(m-1) +\ell}{q\ell}}(t),\]
where $\nu= \frac{ \mu c_{_4} q(m -\ell)}{(q(m-1) + \ell)} $. Therefore, using the equivalence result of $G(t)$ and $E(t)$ in Lemma \ref{Abslem}, and choosing $\mu , c(\epsilon_i)(i=2, 3)$ small enough such that
\begin{align*}
\mu[p - c(\epsilon_2 ) c_2c_{_E} -c(\epsilon_3) ( \epsilon_1 c_{_{\alpha\lambda}}c_{_1} +k_0)] \geq k_3, \\
 \bigl( 1- \epsilon_2  c_2\mu\bigr) - \mu\epsilon_3 ( \epsilon_1 c_{_{\alpha\lambda}}c_{_1} +k_0) c_{_{j\eta}}\geq 0,
\end{align*}
for a positive constant $k_3$. Therefore, we obtain 
\begin{equation}\nonumber
\begin{split}
\Bigl(\frac{1}{\lambda(t)} \mathcal{G}(t) \Bigr)^{\prime} \leq     
 -\frac{k_3}{k_2}\frac{\alpha(t)}{\lambda(t)}\Bigl[\frac{1}{\lambda(t)} \mathcal{G}(t)\Bigr]^{\frac{m}{\ell}}.
 \end{split}
\end{equation}
Define $F(t)$ by 
\[F(t)=: \frac{1}{\lambda(t)} \mathcal{G}(t), \]
hence,we have 
\begin{equation}\label{lastevo}
F^{\prime}(t)\leq -k_4 \frac{\alpha(t)}{\lambda(t)}F^{\frac{m}{\ell}}(t).
\end{equation}
Integrating (\ref{lastevo}) over $(0 , t)$, we obtain 
\begin{equation}\nonumber
F(t)\leq \Bigl[ F^{\frac{\ell -m}{\ell}}(0)+ k_4(\frac{m-\ell}{\ell}) \int^t_0 \frac{\alpha(s)}{\lambda(s)} ds \Bigr]^{\frac{-\ell}{m-\ell}}.
\end{equation}
The equivalence of $F(t)$ and $\mathcal{E}(t)$ gives the desired result.
\end{proof}
Case 2: If  $m=\ell$, then $r=0$ and equation (\ref{cases}) reduces to
\begin{equation}\nonumber
\begin{split}
\Bigl(\frac{1}{\lambda(t)} \mathcal{H}(t) \Bigr)^{\prime}\leq &  
 -\mu \frac{\alpha(t)}{\lambda(t)} \bigl[p - c(\epsilon_2 ) c_2c_{_E} - c(\epsilon_3)( \epsilon_1 c_{_{\alpha\lambda}}c_{_1} +k_0) \bigr] \mathcal{E}(t) \\&+\frac{\alpha(t)}{\lambda(t)} \eta^{\frac{m^\prime}{m}}(t) \Bigl[ \bigl( 1- \epsilon_2  c_2\mu\bigr) - \mu\epsilon_3 ( \epsilon_1 c_{_{\alpha\lambda}}c_{_1} +k_0) c_{_{j\eta}}\Bigr]\mathcal{E}^{\prime}(t)
 \end{split}
\end{equation}
and choosing $\mu , c(\epsilon_i)(i=2, 3)$ small enough as before such that
\begin{align*}
\mu[p - c(\epsilon_2 ) c_2c_{_E} -c(\epsilon_3) ( \epsilon_1 c_{_{\alpha\lambda}}c_{_1} +k_0)] \geq k_3, \\
 \bigl( 1- \epsilon_2  c_2\mu\bigr) - \mu\epsilon_3 ( \epsilon_1 c_{_{\alpha\lambda}}c_{_1} +k_0) c_{_{j\eta}}\geq 0,
\end{align*}
for a positive constant $k_3$, we obtain
\begin{equation}\nonumber
\begin{split}
\Bigl(\frac{1}{\lambda(t)} \mathcal{G}(t) \Bigr)^{\prime} \leq     
 -\frac{k_3}{k_2}\frac{\alpha(t)}{\lambda(t)}\Bigl[\frac{1}{\lambda(t)} \mathcal{G}(t)\Bigr].
 \end{split}
\end{equation}
Setting $F(t)=\frac{1}{\lambda(t)} \mathcal{G}(t)$ as before and integrating the resulting estimate over $(0 , t)$, gives 
\[F(t) \leq F(0) exp\Bigl[-\frac{k_3}{k_2}\int^t_0  \frac{\alpha(s)}{\lambda(s)} ds \Bigr], \]
for $t\geq 0$. Again, since $F(t)\sim \mathcal{E}(t)$, we have the desired estimate. 

\section{Application}
We give examples for which our results apply.\\
The canonical problem  
\begin{equation}\label{examp1}
\bigl(|u_t|^{\ell-2} u_t\bigr)_t -a\nabla\cdot\bigl(|\nabla u|^{q-2} \nabla u\bigr) + b|u_t|^{m-2}u_t =0,
\end{equation}
with boundary conditions $u(t , x)=0 \quad \text{on}\quad \partial \Omega$ and initial conditions
\[u(0 , x)=\psi(x),\quad u_t(0 , x) =\phi(x) \quad \text{for}\quad x\in \Omega.\]
In the bounded domain, $W=[L^{\frac{qn}{n-q}}(\Omega)]^n$, $X=[W^{1,q}_0(\Omega)^n$, $V=[L^\ell(\Omega)]^n$ and \\$Y=[L^m(\Omega)]^n$,
$\alpha(t), \lambda(t), \delta^{\frac{1}{\ell}}(t)$ are constants.
The $C^1$ potential functions are given by
\begin{equation}\nonumber
\begin{split}
q\mathcal{A}(u)=\Vert \nabla u\Vert^q_q=\langle A(u) , u\rangle,\\
\ell \mathcal{P}(v)=\Vert v\Vert^\ell_{\ell} =\langle P(v) , v\rangle
\end{split}
\end{equation} 
and it is easy to show that the conditions given are satisfied.
Generalizing (\ref{examp1}) to 
\begin{equation}\nonumber
\bigl(|u_t|^{\ell-2} u_t\bigr)_t -a\nabla\cdot\bigl(|\nabla u|^{q-2} \nabla u\bigr) +b(t, x) |u_t|^{m-2}u_t=0,
\end{equation}
to an unbounded domain say, the $C^1$ potentials remain the same. We have that
\[ \langle b(t , x)|u_t|^{m-2}u_t , u_t\rangle\geq 0,\]
\[|B(t , x , u_t)|\leq [b(t , x)]^{\frac{1}{m}} [\langle B(t , x, u_t) , u_t\rangle]^{\frac{1}{m^\prime}},\]
where $\langle\cdot , \cdot\rangle$ denotes inner product in $\mathbb{R}^n$ with 
\[ b(t , x)\geq 0 \quad \text{and}\quad \Vert b(t , x\Vert_{\frac{qn}{qn-m(n-q)}} \in L^{\infty}_{loc}(J). \]
Taking $\eta(t)=\Vert b(t , x\Vert_{_{\frac{qn}{qn-m(n-q)}}}$, then its easy to show that
$ \Vert B(t , x, u_t)\Vert_{W^\prime} \leq (\eta(t))^\frac{1}{m} \langle B(t , x , u_t), u_t\rangle^{\frac{1}{m^\prime}}$. Other estimates can be obtained in the same manner.\\
More recently, Ogbiyele and Arawomo\cite{new} considered the nonlinear  problem
\begin{equation}\nonumber
u_{tt} -\Delta u  + b(t , x)|u_t|^{m-1} u_t + |u|^{p-1}u=0,
\end{equation}
where $x\in \mathbb{R}^n$ and showed that for $x\in B(R+t)$ with initial data compactly supported in the ball $B(R)$, the energy decay polynomially for $1\leq m\leq\frac{n+2}{n-2} $, where $W=L^{\frac{2n}{n-2}}(\mathbb{R}^n)$, $X=H^1(\mathbb{R}^n)$, $V=L^2(\mathbb{R}^n)$ and \\$Y=L^{m+1}(\mathbb{R}^n)$. 

Our results also apply to plate equations with $(A= \Delta^2)$,
Observe that, in the case of bounded domains, it is appropriate to choose $W=L^2(\Omega)$ where $\Omega$ is a bounded domian in $\mathbb{R}^n$.

\end{document}